\documentclass[12pt]{amsart}

\usepackage{amsmath}
\usepackage{amsthm}
\usepackage{amsfonts}
\usepackage{amssymb}
\usepackage[all]{xy}
\usepackage{url}
\usepackage{comment}
\usepackage{eucal}
\newcommand\R{\mathbb{R}}

\newcommand\Z{\mathbb{Z}}

\newcommand\pr{\textup{pr}}

\newcommand\acs{\varphi}  				% almost contact structure
\newcommand\scxs{\zeta}					% stable complex structure
\newcommand{\trivR}{\underline{\R}}

\newtheorem{Theorem}{Theorem}[section]

\newtheorem{Proposition}[Theorem]{Proposition}

\theoremstyle{remark}
\newtheorem{Example}[Theorem]{Example}
\newtheorem{Remark}[Theorem]{Remark}
\usepackage{color}

\advance \topmargin by -\headheight
\advance \topmargin by -\headsep
     
\evensidemargin \oddsidemargin
\begin{document}

\title{Contact structures on $M \times S^2$}
\author{Jonathan Bowden}
\address{Mathematisches Institut, Universit\"{a}t Augsburg, Universit\"{a}tstr 14, 86159 Augsburg, Germany}
\email{jonathan.bowden@math.uni-augsburg.de}
\author{Diarmuid Crowley}
\address{Max Planck Institut f\"ur Mathematik, Vivatsgasse 7, D-53111 Bonn, Germany}
\email{diarmuidc23@gmail.com}
\author{Andr\'{a}s I. Stipsicz}
\address{ R\'{e}nyi Institute of Mathematics, Re\'{a}ltanoda u. 13-15., Budapest, Hungary H-1053}
\email{stipsicz@renyi.hu}

\begin{abstract}
	We show that if a manifold $M$ admits a contact structure, then so does $M \times S^2$.
	Our proof relies on surgery theory, a theorem of Eliashberg on contact surgery and a theorem of Bourgeois 	showing that if $M$ admits a contact structure then so does $M \times T^2$.
\end{abstract}

\maketitle

%\begin{abstract}
%	We show that if a manifold $M$ admits a contact structure, then so does $M \times S^2$.
%	Our proof relies on surgery theory and a result of Bourgeois showing that if $M$
%	admits a contact structure then so does $M \times T^2$.
%\end{abstract}

\section{Introduction} \label{sec:introduction} % 
%%%%%%%%%%%%%%%%%%%%%%%%%%%%%%%%%%%%%%%%%%%%%%%%%%%%%%
One of the most important questions in contact topology is to
determine which odd dimensional oriented manifolds admit contact structures.
Recall that a (positive, coorientable) contact structure on an
oriented manifold $M$ of dimension $2q+1$ is a hyperplane distribution
$\xi \subset TM$ which can be given as $\ker \alpha $ for a 1-form
$\alpha \in \Omega ^1 (M)$ satisfying
\[
\alpha \wedge (d\alpha )^q >0.
\]
The 2-form $d \alpha$ defines a symplectic form on $\ker \alpha$,
which determines an almost complex structure $J$ on the sub-bundle $\xi \subset TM$,
unique up to contractible choice.  
Therefore the existence of a contact structure implies that $TM$
decomposes as the sum of a $q$-dimensional complex bundle and a trivial
real line bundle. 
The pair $(\xi \subset TM, J)$ is called an {\em almost contact structure} on $M$.
It is equivalent to a reduction of the structure group of $TM$ from $SO(2q+1)$ to $U(q)\times 1$. 
%
%Such a reduction is called an \emph{almost contact structure} on $M$. 
%
Now the above existence question can be refined as follows: Which almost contact 
manifolds admit contact structures?

The answer to this question is positive for open manifolds (by an
application of Gromov's $h$-principle), in dimension three (by Lutz
\cite{Lutz} and Martinet \cite{Martinet}) and in dimension five (by
Casals-Presas-Pancholi \cite{CPP} and Etnyre \cite{Etnyre}).  
(For further results see \cite{geiges}.) 
Less is
known for higher dimensional closed manifolds, but so far no example of
an almost contact manifold with no contact structure has been found.
 According to a beautiful
result of Bourgeois \cite{bourgeois}, for a closed contact manifold $(M, \xi)$
the product $M\times \Sigma _g$ also admits a contact structure
provided $g\geq 1$. (Here $\Sigma _g$ denotes the 
 closed orientable surface of genus $g$.) This construction relies on
the theory of compatible open book decompositions of Giroux-Mohsen \cite{Giroux},
and provides a contact structure on $M\times T^2$ with the property that
for each $p\in T^2$ the submanifold $M\times \{ p\}\subset M\times T^2$
is contact and indeed contactomorphic to $(M, \xi)$.

The purpose of the present article is to prove that the result of Bourgeois
holds for $g=0$ as well. (The $g=0$ case is expected to play a key role in 
the positive resolution of the  general existence problem.)

\begin{Theorem}\label{thm:main}
%%%%%%%%%%%%%%%
Suppose that $(M, \xi)$ is a closed, contact manifold. Then, 
the product $M\times S^2$ admits a contact structure.
\end{Theorem}
\noindent Furthermore, with a little more care we prove a
relative version of this result, which answers a question posed by
F. Presas:
\begin{Theorem}\label{thm:main_rel}
Suppose that $(M, \xi)$ is a closed, contact manifold and let $p \in
S^2$. Then the product $M\times S^2$ admits a contact structure such
that the submanifold $M\times \{ p\}$ is contact and the natural map
to $(M, \xi)$ is a contactomorphism.
\end{Theorem}
%\noindent In fact, essentially the same argument yields contact structures 
%on the product of $M$ with any even dimensional sphere $S^{2k}$.

\begin{Remark} \label{rem:generalisation}
The proofs of Theorems \ref{thm:main} and \ref{thm:main_rel} generalise so that we may replace
$S^2$ with other even-dimensional manifolds including any
even-dimensional sphere; see Theorem \ref{thm:further}.
\end{Remark}
\begin{Remark} \label{rem:HW}
Theorem \ref{thm:main} has also been obtained independently in \cite{HW} using different methods.  
\end{Remark}

The idea of the proof is the following: fix a contact structure on
$M^{2q+1}$ and consider the contact structure on $M\times T^2$
provided by the construction of Bourgeois. Let the corresponding
almost contact structure be denoted by $\varphi$. Then we claim that
there is a smooth $(2q+4)$-dimensional cobordism $Y$ from $M\times T^2$ to $M\times S^2$
which admits an almost complex structure extending $\varphi$ and a
Morse function with critical points of indices $\leq q+2$.  By work of
Eliashberg and Weinstein \cite{Eliashberg, Weinstein}, for $q\geq 1$
such a cobordism gives rise to a sequence of contact surgeries on $M
\times T^2$, inducing a contact structure on $M\times S^2$.  The
existence of the cobordism $Y$, on the other hand, can be naturally
studied in the framework of stable complex surgery.
%, a point of view pursued in detail in \cite{BCS}.

\bigskip

{\bf {Acknowledgements:}} The authors would like to thank the Max
Planck Institute for Mathematics in Bonn for its hospitality while
parts of this work has been carried out, and Hansj\"org Geiges for
useful comments on an earlier draft of the paper.  AS was partially supported
by OTKA NK81203, by the \emph{Lend\"ulet program} of the Hungarian
Academy of Sciences and by ERC LDTBud.
The present work is part of the authors' activities within CAST, a
Research Network Program of the European Science Foundation.

\section{Preliminaries} \label{sec:complex}
%%%%%%%%%%%%%%%%%%%%%%%%%%%%%%%%%%%%%%%%%%%%%%%%%%%%%%
Let $\alpha$ be a contact form on a closed $(2q+1)$-manifold $M$ with
associated contact structure $\xi$ and let $\acs$ be the induced
almost contact structure. Then $\varphi$
naturally induces a stable complex structure $\zeta_\acs$ on the
stable tangent bundle $\tau_M : = TM \oplus \trivR^k$ of $M$ (where
$\trivR^k$ denotes the trivial real $k$-plane bundle).

Define the manifold $W$ as $(D^2 \times S^1)-{\rm Int}(D^3)$, the
solid 3-dimensional torus with a small open 3-ball removed from its
interior.  Observe that $W$ admits a map $c \colon W\to S^2$ which is
a diffeomorphism on the boundary component $\partial _1 W$
diffeomorphic to $S^2$ (and
has degree one on the boundary component 
$\partial _0 W$
diffeomorphic to $T^2$).

Consider the stable complex structure $\zeta$ on the solid torus
which comes from the splitting $T(D^2 \times S^1) = TD^2 \times TS^1$ and from a
choice of an almost complex structure on $TD^2$.  Let $\zeta_W:= \zeta|_{W}$ denote
the restriction of $\zeta$ to $W$, and let $\zeta _{S^2}$ and $\zeta
_{T^2}$ denote the stable complex structures induced by the complex 
structures on $S^2$ and $T^2$. Note that by construction $\zeta
_{T^2}$ is homotopic to the stabilisation of an almost complex structure on $T^2$. Since $W \simeq S^1 \vee S^2$, after a choice of trivialisation homotopy classes of stable complex
structures on $TW$ can be identified with $
H^2(S^1\vee S^2; \pi _2(SO/U))\cong \pi_2(SO/U)$; in particular, in follows that 
the pull-back stable complex structure $c^*(\scxs _{S^2})$ is
homotopic to $\scxs_W$.
% ; written $c^*(\scxs_{S^2}) \simeq \scxs_W$.
%Viewing $W$ as a cobordism from $T^2$ to $S^2$, standard orientation conventions entail that $\scxs_W|_{T^2} = -\scxs_{T^2}$.

We now consider a general stably complex $n$-manifold $(X, \scxs_X)$.
(In our immediate applications $(X, \zeta _X)$ 
will be either  $(W, \zeta _W)$, 
$(S^2, \zeta _{S^2})$ or  $(T^2, \zeta _{T^2})$; in Section~\ref{sec:further}
we will consider more general stably complex manifolds.)
The stable tangent bundle of $M \times X$ is the exterior Whitney sum 
\[ \tau_{M \times X} = \tau_M \times \tau_X , \]
and therefore it admits the stable complex structure $\zeta_\acs
\times \zeta_X$.  In particular,  the products $M\times W$,
$M\times S^2$ and $M\times T^2$
 admit the stable
complex structures $\zeta _{\acs}\times \zeta _W$,
$\zeta _{\acs}\times \zeta _{S^2}$ and $\zeta _{\acs}\times \zeta _{T^2}$.
We will view $M\times W$ as a cobordism from $M\times T^2$ to $M\times S^2$ 
so that as oriented manifolds 
\[ \partial (M \times W) = -(M \times T^2) \sqcup (M \times S^2),\]
where the orientations on all manifolds are those induced by the stable complex structure on $T(M \times W)$.
\begin{comment}
\[ (\scxs_\acs \times \scxs_W)|_{M \times T^2} = -\scxs_\acs \times \scxs_{T^2} \quad \text{and} \quad 
(\scxs_\acs \times \scxs_W)|_{M \times S^2} = \scxs_\acs \times \scxs_{S^2}.   \]
\end{comment}
%

The map $g\colon M\times W \to M\times S^2$ given by ${\rm {id}}_M\times c$ is covered by a map of stable tangent bundles $\bar g \colon
\tau_{M \times W} \to \tau_{M \times S^2}$, and $g^*(\zeta
_{\acs}\times \zeta _{S^2})$ is homotopic to 
%$\zeta _{\acs}\times c^*(\zeta _{S^2}) \simeq 
$\zeta _{\acs}\times \zeta_W$. The above manifolds and maps fit in the commutative diagram
\begin{equation}\label{eq:triangle}
\xymatrix{ M\times T^2 \ar[dr]_(0.45){f_0} \ar[r]^{i_{0}} & M\times W 
\ar[d]^(0.4){g} & M\times S^2 \ar[l]_(0.45){i_{1}} \ar[dl]^(0.45){f_1} \\ & M\times S^2, }
\end{equation}
where the maps $i_0,i_1$ are the embeddings of the boundary
components, $f_1={\rm {id}}_M\times c|_{\partial _1W}$ is a diffeomorphism and
$f_0={\rm {id}}_M\times c|_{\partial _0 W}$ has degree one.  In addition, the bundle
  map $\bar g$ above restricts to give bundle maps $\bar f_0 \colon
  \tau_{M \times T^2} \to \tau_{M \times S^2}$ and $\bar f_1 \colon
  \tau_{M \times S^2} \to \tau_{M \times S^2}$ covering $f_0$ and
  $f_1$ respectively. (As always, a bundle map is an
  isomorphism of a bundle with the pullback of the target bundle.)

\section{A contact structure on  $M\times S^2$} \label{sec:proof}

In this section we prove Theorems \ref{thm:main} and \ref{thm:main_rel};
the proofs will be simple consequences of Propositions \ref{prop:first}, \ref{prop:second}, 
\ref{prop:bour} and \ref{prop:wein} below.

Our first proposition, Proposition \ref{prop:first}, is an analogue of Kreck's \cite[Proposition 4]{Kreck99}.  Whereas Kreck works with bundle maps from the stable normal
  bundle, we work with bundle maps from the stable tangent bundle, since this better reflects the
  contact geometry involved.
  The modifications from stable normal surgery to stable tangential surgery are standard: for
  example, stable tangential surgery is treated in \cite[Theorem 3.59]{Lueck} in the case where the target of the
  surgery is a Poincar\'{e} pair.  However, the techniques for surgery below the middle dimension,
  which are all that we use, do not rely on the target being a Poincar\'{e} pair.
  For the sake of completeness, we give the proof which involves making minor modifications to the 
  proof of \cite[Proposition 4]{Kreck99} which arise in the stable tangential setting.  Recall that $M$ is a closed smooth $(2q+1)$-dimensional
manifold, hence $M\times W$ is a compact $(2q+4)$-manifold with boundary.
  
\begin{Proposition}\label{prop:first}
  The manifold $M\times W$ can be modified by a finite sequence of
  surgeries in its interior to obtain a manifold $Y$ with the
  following properties:
\begin{itemize}
\item $Y$  fits into the following commutative diagram:
\begin{equation}\label{eq:trianglewithy}
\xymatrix{ M\times T^2 \ar[dr]_(0.45){f_0} \ar[r]^(0.6){i_{0}} & Y 
\ar[d]^(0.4){g_Y} & M\times S^2 \ar[l]_(0.6){i_{1}} \ar[dl]^(0.45){f_1} \\ & M\times S^2.} 
\end{equation}
\item The map $g_Y$ is a $(q+2)$-equivalence, that is,
$(g_Y)_* \colon \pi _i (Y)\to \pi _i(M\times S^2)$
is an isomorphism  for $i\leq q+1$ and a surjection for $i=q+2$.
\item There is a bundle map $\bar g_Y \colon \tau_Y \to \tau_{M \times S^2}$ covering $g_Y$
which restricts to the bundle maps $\bar f_0$ and $\bar f_1$ on the boundary of $Y$.
Hence $Y$ admits a stable complex structure $\scxs_Y$ such that
$(Y, \scxs_Y)$ is a stable complex bordism from $(M \times T^2, \scxs_{\acs} \times  \scxs_{T^2})$ to $(M \times S^2, \scxs_{\acs} \times  \scxs_{S^2})$.
\end{itemize}
\end{Proposition}

\begin{proof}
 Let  $B := M\times S^2$, let $\tau_{B} \colon M \times S^2 \to BO$ be the classifying map of
  the stable tangential bundle of $B$ and let $g \colon M \times W \to M \times S^2$ be the map described in Section \ref{sec:complex}.
  
  We proceed by induction on homotopy groups $\pi_i$ starting from $g \colon M \times W \to M \times S^2$. Since both $M\times W$ and 
  $M\times S^2$ are connected, we have an isomorphism for $i=0$. 
  Let $\pi = \pi_1(M \times S^2) = \pi_1(M)$.  Note that $g_* \colon \pi_1(M \times W) \to \pi_1(M \times S^2)$
  is isomorphic to the projection $\pi \times \Z \to \pi$, hence $g_*$ is surjective on $\pi_1$.
Now consider the following commutative diagram:
% \begin{equation}\label{eq:trianglewith}
%\xymatrix{ &M\times S^2=B \ar[d]^{\tau_B} \ar[dr]^{\zeta_B}&\\
%X \ar[ur]^{g_X} \ar[r]^{\tau_X} & BO & \ar[l]_F BU}
%\end{equation}
\begin{equation}\label{eq:trianglewith}
\xymatrix{ &M\times S^2=B \ar[d]^{\tau_B} \\
X \ar[ur]^{g_X} \ar[r]^{\tau_X} & BO}
\end{equation}
where $X$ is a bordism from $M \times T^2$ to $M \times S^2$.
Suppose that the map $g_X$ induces an isomorphism between the homotopy groups 
 $\pi_j (X)\to \pi_j (B)$ for $j<i\leq q+1$ and a surjection on $\pi_i$. 
 We first kill the kernel of $(g_X)_* \colon \pi_i(X) \to \pi_i(B)$.
 Since $\pi_j (B, g_X(X))=0$ for $j<i$, by the Hurewicz Theorem we have 
 that $\pi_i (B, g_X (X)) \cong H_i (B, g_X (X); \Z[\pi])$, hence the 
 kernel of $(g_X)_*$ is finitely generated over $\Z[\pi]$.
 Suppose that $S^i \to X$ represents a generator of the kernel of $(g_X)_*$.
 For dimensional reasons we can assume that 
 $S^i$ is embedded. 
  For any $i$ the stable tangent bundle $\tau _{S^i}$ 
  is stably trivial and $\tau _X|_{S^i}$ is the pull-back 
 from $B$ along a homotopically trivial map, hence
$\tau_X|_{S^i} = \nu_{S^i\subset X} \oplus \tau_{S^i}$ implies that  the normal bundle $\nu_{S^i\subset X} $ of $S^i$ in $X$ is stably trivial. 
Since the rank of $\nu_{S^i \subset X}$ is greater than $i$,
 if follows that $\nu_{S^i \subset X}$ is trivial.
 In order to kill the class represented by $S^i$, we attach a $(2q+5)$-handle 
 to $X$ along $D^{m-i} \times S^i \subset X$,
 where $m : = 2q+4$. 
 For a particular choice of framing the map $g_{X}$ will extend over the attached handle in such a 
 way that   
 analogue of diagram \eqref{eq:trianglewith} above for the induced cobordism remains 
 commutative 
 \cite[Lemma 2 (ii)]{Kreck99}; that is, the bundle map $\bar g_X$ extends over the trace of the 
 surgery to classify the stable tangent bundle of this trace.
 Since we are free to choose the framing, we choose this particular one. After finitely many 
 surgeries we can kill the kernel on $\pi_i$ and maintain the stable tangential bundle maps.
 
Now we must arrange that the map $g_X$
is surjective on $\pi_{i+1}$ for $i \geq 1$.
Since $B$ is a finite $CW$-complex, $\pi_{i+1}(B)$ is finitely generated over $\Z[\pi]$.
 For each element $x_j$  of a generating set $\{x_1, \dots, x_k \}$ of the 
 cokernel of $(g_X)_* \colon \pi_{i+1}(X) \to \pi_{i+1}(B)$, 
 we consider a twisted bundle $S^{m-i}{\widetilde {\times}}_{\alpha_j} \thinspace S^{i+1}$, where
  the twisting $\alpha_j$ is determined by the image of $(g_X)_*(x_j)$ in $\pi_{i+1}(BO)$. 
  The map $g_X$ can be extended from $X$ to the interior connected sum of $X$ 
  with this twisted bundle in such a way that the commutativity of diagram~\eqref{eq:trianglewith}
  is preserved. As a result, we obtain a new map $g_X \colon X \to B$  such that $g_{X}$ induces a surjective map on 
  $\pi _{i+1}(X)$ and is covered by a map of the stable tangent bundle of $X$.
  Inductively repeating this procedure for $i\leq q+1$ we obtain a manifold 
  $Y$ and a map $g_Y\colon Y \to M\times S^2$ with the desired properties.
\end{proof}

Consider now the cobordism $Y$ between $M\times T^2$ and $M\times S^2$ given by Proposition \ref{prop:first}.
\begin{Proposition}\label{prop:second}
For $q\geq 1$ the cobordism $Y^{2q+4}$ admits a 
handle decomposition with handles of index at most $q+2$ attached to $(M\times T^2) \times [0,1]$.
\end{Proposition}

\begin{proof}
In the terminology of \cite{Wall} we shall show that 
$Y$ (as a cobordism built on $M\times S^2$) is \emph{geometrically}
$(q+1)$-connected. Indeed, according to \cite[Theorem~3]{Wall}, this property 
follows once we can show that the cobordism is $(q+1)$-connected, that is,
the relative homotopy groups $\pi _i (Y, M\times S^2)$ vanish for
$i\leq q+1$. Notice, however, that the portion
\[
\xymatrix{  & Y 
\ar[d]^{g_Y} & M\times S^2 \ar[l]_{i_{1}} \ar[dl]^{f_1} \\ & M\times S^2} 
\]
of diagram~\eqref{eq:trianglewithy} implies that $(i_1)_* $ is an
isomorphism for $i\leq q+1$, since $(g_Y)_*$ is an isomorphism in all
these dimensions, and $f_1$ is a diffeomorphism.  The long exact
sequence of homotopy groups for the pair $(Y, M\times S^2)$ shows that
in the dimensions $i\leq q+1$ we have vanishing relative homotopy
groups and so by \cite[Theorem~3]{Wall}, $Y$ 
has a handle decomposition relative to $M \times S^2$
with handles of index $q+2$ and higher.  Hence $Y$ has a handle decomposition relative to 
$M \times T^2$ with handles of index at most $q+2$.
\end{proof}

\begin{Proposition}\label{prop:bour}
Suppose that $(M, \xi )$ is a contact manifold
and let $\pr _1\colon M\times T^2\to M$ and $\pr _2\colon M\times T^2\to T^2$
denote the projections. Then there is a
contact structure $\xi '$ on $M\times T^2$ such that, the induced %underlying% 
almost contact structure is homotopic in $T(M\times T^2)$ to the complex sub-bundle 
$\pr _1^*(\xi, J)\oplus \pr _2^*(T(T^2),J_{T^2})$, where $J_{T^2}$ 
is an almost complex structure on $T^2$. In particular, the stable complex structure induced by 
$\xi'$ is homotopic to the stable complex structure 
$\scxs_\acs \times \scxs_{T^2}$ of Section \ref{sec:complex}.
\end{Proposition}
\begin{proof}
In \cite{bourgeois} a contact structure on $M\times T^2$
was given by the following formula: if $\xi $ on $M$ is defined as
$\ker \alpha$ then we define
\[
\alpha '=\pr _1^*(\alpha) + f(r)(\cos \theta dx_1+\sin \theta dx_2),
\]
where $\theta$ is the angular coordinate coming from an open book
decomposition of $M$ compatible with $\alpha$, $r$ is the radial
coordinate in a small neighbourhood of the binding of the open book, $f$ is a suitable
function and $x_1, x_2$ are coordinates on $T^2$.  
(For further details of this construction see \cite{bourgeois}.)
The contact
structure $\xi '=\ker \alpha '$ intersects the sub-bundle $\pr_1^*(TM)$
in $\pr _1^*(\xi)$, therefore as symplectic vector bundles
\begin{equation} \label{eq:symplectic}
(\xi ',d \alpha'|_{\xi'}) \cong \left(\pr _1^*(\xi), \ \pr _1^*\thinspace(d \alpha)\right) \oplus (E,d \alpha'|_E) ,
\end{equation}
where $ E = \pr _1^*(\xi )^{\perp d\alpha '}$ denotes the symplectic
complement of $\pr _1^*(\xi )$ with respect to the symplectic form
$d\alpha'|_{\xi'}$. Since projection to the second factor maps
$\pr_1^*(\xi )^{\perp d\alpha '}$ to the trivial sub-bundle tangent to
the $T^2$ fibers, we obtain the splitting described in
\eqref{eq:symplectic} above. This symplectic splitting then determines a product complex structure
whose stabilisation is homotopic to $\scxs_\acs \times \scxs_{T^2}$.
\end{proof}

The final ingredient we need in the proof of Theorem~\ref{thm:main} is the following result of Eliashberg~\cite{Eliashberg}, which realises certain cobordisms via what have become known as Weinstein handle attachments~\cite{Weinstein}.
\begin{Proposition}[\cite{Eliashberg}]\label{prop:wein}
Suppose that $(Y,J)$ is a compact $(2q+2)$-dimensional almost complex cobordism from $M_1$ to $M_2$, where $M_1$ and $M_2$ are closed manifolds. Suppose furthermore that 
$q\geq 2$ and $Y$ admits a handle 
decomposition with handles of indices $\leq q+1$, and $M_1$ admits a 
contact structure with induced almost contact structure being equal to the restriction of
$J$ along $M_1$. Then, the manifold $M_2$ admits a contact structure. \qed
\end{Proposition}

With these preparatory results at our disposal, we now turn to the
proofs of Theorems~\ref{thm:main} and \ref{thm:main_rel}.
\begin{proof}[Proof of Theorem~\ref{thm:main}]
For $q=0$ the manifold $M$ is diffeomorphic to $S^1$, and $S^1\times S^2$ is 
known to admit a contact structure. 
Consider now a closed contact manifold $(M, \xi )$ of dimension $2q+1\geq 3$,
and apply the result of
Bourgeois~\cite{bourgeois} to equip $M \times T^2$ with a contact
structure.  By Proposition~\ref{prop:bour} the almost contact structure
induced by this contact structure on $M \times T^2$ is $ \scxs_{\acs} \times  \scxs_{T^2}$ considered in Section~\ref{sec:complex}.

Now consider the cobordism $Y$ given by Proposition~\ref{prop:first}. The stable complex
structure may be destabilised to an almost complex structure, which extends
the almost contact structure given by the the contact structure constructed by Bourgeois 
on $M\times T^2$ (and is described in Proposition \ref{prop:bour}). Indeed, since $Y$
is given by attaching handles of index at most $q+2$ to $M\times T^2$,
the obstruction for extending the stable complex structure
from $M\times T^2$ coincides with the obstruction for extending the
(unstable) complex structure. (This last claim follows from
the fact that the embedding $SO(2q+4)/U(q+2)\to SO/U$ induces isomorphisms on the
homotopy groups of dimensions $i\leq q+1\leq 2q+2$.)

By Proposition~\ref{prop:second}
the cobordism $Y$ has a handle decomposition with handles of
index at most $ q+2$ (as a cobordism built on $M\times T^2$) and it admits an almost complex structure extending
the one on $M\times T^2$ supporting a contact structure, 
therefore Proposition~\ref{prop:wein} implies the claimed existence result.
\end{proof}
\begin{proof}[Proof of Theorem~\ref{thm:main_rel}]
To prove the relative case we consider the cobordism $\widehat{W}$
given by removing a ball from the solid torus $D^2 \times S^1$ and
removing an open neighbourhood of an embedded arc joining the boundary
components and intersecting $S^2$ in the point $p$. The product $M
\times \widehat{W}$ is then a cobordism between $M \times (S^2 - D^2)$ 
and $M \times (T^2 - D^2)$ and there is a
natural map $\widehat{W} \to (S^2 - D^2)$ that is a
diffeomorphism on the boundary component corresponding to $S^2 - D^2$.

Then the same argument shows that Proposition~\ref{prop:first}
holds when $S^2$ and $T^2$ are replaced by $S^2 - D^2$ and $T^2
- D^2$, respectively. Moreover, the results of Wall
\cite{Wall} apply to cobordisms between manifolds with boundary, when the cobordism between the boundaries of the boundary manifolds is
a product, which is the case for the manifold $\widehat{Y}$ that is obtained
from $M \times \widehat{W}$ via surgery as in Proposition
\ref{prop:first}. Thus the argument of Proposition \ref{prop:second}
applies and we conclude that
$M \times (S^2 - D^2)$ can be obtained from $M \times (T^2 - D^2)$ via handle
attachments of index at most $q+2$. Now we glue in a copy of $M \times D^2 \times [0,1]$ along part of the boundary of $\widehat{Y}$ to obtain a bordism from $M \times T^2$ to $M \times S^2$. In order to realise these handle
attachments via Weinstein handles~\cite{Weinstein}, one must first apply an
$h$-principle to find an isotopy of the spheres to attaching 
spheres that are isotropic. Since
this can be done in a $C^0$-small fashion
(cf.\ \cite{Cieliebak&Eliashberg12}), we see that all the Weinstein
handles can be attached along spheres that are disjoint from $M\times
\{ p\}$, where $ p \in D^2 \subset S^2$. Finally, the contact structure found by Bourgeois on $M\times
T^2$ has the additional property that for any $p\in T^2$ the
submanifold $M\times \{ p \}$ is contact and contactomorphic to $(M, \xi )$. Since contact surgery preserves the contact structure outside a small neighbourhood of the surgery sphere, the result follows.
\end{proof}
%(the formal condition for the existence of such an isotopy is a fulfilled since the almost complex structure extends over each $i$ handle and $i \leq q+2$)
\section{Final remarks}
\label{sec:further}
We point out that one can actually show
that all almost contact structures on $M\times S^2$ inducing 
the stable complex structure $\scxs_\acs \times \scxs_{S^2}$ 
admit contact structures: the details will appear in \cite{BCS}.  
%In addition, when starting our construction
%we can choose any contact structures on $M$, enhancing the almost
%contact structures on $M\times S^2$ for which the above argument provides
%a contact structure.
%
In addition, the arguments used to
prove the existence of a contact structure on $M \times S^2$ actually
show the following:
\begin{Theorem}\label{thm:further}
Suppose that $(M, \xi )$ is a closed contact manifold inducing the stably complex manifold 
$(M, \scxs_\acs)$, and that  
$(X^{2k}, \scxs_X)$ is a closed stably complex manifold. Suppose furthermore
that  $(X, \scxs_X)$ satisfies the following conditions:
\begin{itemize}
\item There is a closed stably complex manifold 
$(Z^{2k}, \scxs_Z)$ and a stably complex cobordism $(W, \scxs_W)$ between $(X, \scxs_X)$ and $(Z, \scxs_Z)$ 
which admits a map 
$c\colon W\to X$ restricting to a diffeomorphism along $X\subset W$.
\item The product manifold $M \times Z$ admits a contact structure $\xi'$ compatible with $\scxs_\acs \times \scxs_Z$. 
%\item $M\times W$ admits a stable complex structure which restricts to the stable 
%complex structure on $M\times Z \subset M\times W$
%induced by $\xi '$.
\end{itemize}
Then $M\times X^{}$ admits a contact structure compatible with $\scxs_\acs \times \scxs_X$.
Moreover, the contact structure on $M\times X$ can be chosen 
in such a way that for a fixed point $x\in X$ the submanifold $M\times \{ x\}$ is a
contact submanifold and the natural projection restricted to 
$M\times \{ x\}$ is a contactomorphism to $(M, \xi )$. \qed
\end{Theorem}

\begin{Example}
It is not hard to see that the manifold $X=S^{i_1}\times \ldots \times S^{i_n}$ with 
$\sum _{j=1}^ni_j=2k$ satisfies the conditions of Theorem~\ref{thm:further}, when one further chooses
$Z^{2k}$ to be the $2k$-dimensional torus $T^{2k}$.
\end{Example}

%\begin{Remark} For the special case that $M$ is a product of spheres, the existence of a contact structure on $M\times S^2$ was was already shown in  \cite{HW}.
%\end{Remark}

\end{document}